\documentclass[11pt,reqno]{amsart}

\usepackage{amsmath}
\usepackage{amssymb}
\usepackage{amsthm}
\usepackage{enumerate}
\usepackage[mathscr]{eucal}

\usepackage{mathtools}
\usepackage{microtype}

\usepackage[english]{babel}

\usepackage{amsmath,amssymb,latexsym,textcomp,mathrsfs}
\usepackage[all]{xy}
\usepackage{graphicx}
\usepackage{bm,amsmath, amsthm, amssymb, amsfonts}
\usepackage{times}
\usepackage[english]{babel}

\setbox0=\hbox{$+$}
\newdimen\plusheight
\plusheight=\ht0
\def\+{\;\lower\plusheight\hbox{$+$}\;}  

\setbox0=\hbox{$-$}
\newdimen\minusheight
\minusheight=\ht0
\def\-{\;\lower\minusheight\hbox{$-$}\;}

\setbox0=\hbox{$\cdots$}
\newdimen\cdotsheight
\cdotsheight=\plusheight
\def\cds{\lower\cdotsheight\hbox{$\cdots$}}

\numberwithin{equation}{section}
\newtheorem{thm}{Theorem}[section]

\newtheorem{prop}[thm]{Proposition}

\newtheorem{cor}{Corollary}

\theoremstyle{definition}
\newtheorem{defn}{Definition}[section]

\newtheorem{exmp}{Example}[section]

\theoremstyle{remark}
\newtheorem{rem}{\bf{Remark}}
\newtheorem{note}{\bf{Note}}
\numberwithin{equation}{section}

\begin{document}

\setcounter {page}{1}
\title{On rough continuity and rough $I$-continuity of real functions}

\author{Amar Kumar Banerjee and Anirban Paul}

\address[A.K.Banerjee]{Department of Mathematics, The University of Burdwan, Golapbag, Burdwan-713104, West Bengal, India.}

\email{ akbanerjee1971@gmail.com, akbanerjee@math.buruniv.ac.in }

\address[A.Paul]{Department of Mathematics, The University of Burdwan, Golapbag, Burdwan-713104, West Bengal, India.}
         
\email{ paulanirban310@gmail.com}

\begin{abstract} 
In this paper, we have studied first the idea of rough continuity of real valued functions of real variables and then we have discussed some important properties of rough continuity. Then we study the idea of rough $I$-continuity of real valued functions and find the relation  between rough $I$-continuity and rough continuity. We also introduce the notion of rough $(I_1, I_2)$-continuity and rough $I^*$-continuity of real valued functions and discuss some properties on this two types on continuity.
\end{abstract}
\maketitle
\author{}


\textbf {(2020) AMS subject classification :} 54A20, 40A35. \\
\section{Introduction}
The notion of statistical convergence of sequences of real numbers was given independently by Fast \cite{f} and Steinhaus \cite{q} as a generalization of ordinary convergence. Then over the years lot of development were made in this area (see \cite{g,p}). The notion of $I$-convergence of the sequences of real numbers which is a generalization of the notion of statistical convergence was given by Kostyrko et. al. \cite{h} using the structure of the ideal $I$ of subsets of the set of natural numbers. An another type of convergence which is closely related to the ideas of $I$-convergence is the idea of $I^*$-convergence given by Kostyrko et. al. \cite{4th4}. It is seen in \cite{h} that these notions are equivalent if and only if the ideal satisfies the property (AP). Several works were done in recent years on $I$-convergence (see \cite{3,5,ownpaper,6,5thpaper,i,j,m}).\\
The notion of rough convergence of sequences in a finite dimensional space was introduced by Phu \cite{R2} in 2001. In 2013, S. K. Pal et. al. \cite{janaref} introduced the notion rough ideal convergence using the concepts of $I$-convergence and rough  convergence and later in 2014, E. D\"undar et. al. \cite{R7} gave the idea of of rough $I$-convergence in normed linear spaces independently. In \cite{p5/1}, V. Bal\'a\u{z} et. al. introduced the notion of $I$-continuity of functions $f: \mathbb{R}\mapsto \mathbb{R}$ as a generalization of the statistical continuity of the functions introduced by J. \u{C}erve\u{n}ansk\'y in \cite{p5/2}. In their paper, V. Bal\'a\u{z} et. al. \cite{p5/1} introduced the notion of $I$-continuity, $(I_1, I_2)$-continuity and $I^*$-continuity of real valued functions, where $I$, $I_1$, $I_2$ are the ideals on the set of natural numbers. It is observed from \cite{p5/1} that $I$-continuity and ordinary coincide for any admissible ideal $I$. This motivates us to investigate where such result continue to hold for the case of rough $I$-continuity.\\
In our present work, using the concept of rough convergence of real sequences and $I$-continuity of real valued functions we have introduced the notion of rough $I$-continuity, rough $(I_1, I_2)$-continuity and rough $I^*$-continuity of real valued functions of real variables. Then we have discussed some properties of these ideas and have found out relations between them. Also we have further verified how far some of the results which are true in case of $I$-continuity, $(I_1, I_2)$-continuity, $I^*$-continuity continue to hold in case of rough $I$-continuity, rough $(I_1, I_2)$-continuity and rough $I^*$-continuity, where $I$, $I_1$, $I_2$ are ideals on the set of natural numbers. Likely in the case of continuity and $I$-continuity we have shown that the concept of rough continuity and rough $I$-continuity are equivalent for any non trivial admissible ideals $I$.\\
Before going to the main results we will recall some basic definitions and notions which will be needed in sequel.
\section{Preliminaries}
Throughout our discussion $\mathbb{R}$, $\mathbb{N}$ will denotes the set of all real numbers and set of all natural numbers respectively. $I$, $I_1$, $I_2$ denote the non trivial admissible ideals of the set of $\mathbb{N}$ unless otherwise stated.
\begin{defn}\cite{f}
Let $K$ be a subset of the set of natural numbers $\mathbb{N}$ and let us denote the set $K_i =\{k\in K : k\leq i\}$. Then the natural density of $K$ is given by $d(K)=\displaystyle{\lim_{i\rightarrow \infty}}\frac{|K_i|}{i}$, where $|K_i|$ denotes the cardinality of the set $K_i$. 
\end{defn}
\begin{defn}\cite{f}
A sequence $\{x_n\}_{n\in\mathbb{N}}$ of real numbers is said to be statistically convergent to $x$ if for any $\varepsilon >0$, $d(A(\varepsilon))=0$, where $A(\varepsilon)=\{n\in \mathbb{N}: |x_n - x|\geq \varepsilon\}$.
\end{defn}
Let $I$ be a collection of subsets of a set $S$. Then $I$ is called an ideal on $S$ if $(i)$ $A, B\in I$ $\Rightarrow A\cup B\in I$ and $(ii)$ $A\in I$ and $B\subset A$ $\Rightarrow B\in I$ \cite{t}.\\
An ideal $I$ on $S$ is called admissible if it contains all the singletons, that is, $\{s\}\in I$ for each $s\in S$. $I$ is called nontrivial if $S\notin I$ and $I\neq \phi$ \cite{t}. From the definition it follows that $\phi\in I$.\\
If $S=\mathbb{N}$, the set of all positive integers then $I$ is called an ideal on $\mathbb{N}$. We will denote by Fin the ideal of all finite subsets of a given set $S$. \\
If $I$ is a non trivial ideal on $S$, then the class $F(I)=\{M\subset \mathbb{N}: \text{there exists } A\in I\:\: \text{such that }\: M=\mathbb{N}\setminus A\}$ is a filter on $S$, called the filter associated with $I$.

\begin{defn}\cite{h}
An admissible ideal $I\subset 2^\mathbb{N}$ is said to satisfy the condition (AP) if for every countable family of mutually disjoint sets $\{A_1, A_2,\cdots\}$ belonging to $I$ there exists a countable family of sets $\{B_1, B_2,\cdots\}$ such that the symmetric difference $A_j\Delta B_j$ is a finite set for each $j\in\mathbb{N}$ and $B=\displaystyle{\bigcup_{j=1} ^ \infty} B_j \in I$. Several example of countable family satisfying (AP) are seen in \cite{h}.
\end{defn}
\begin{defn}\cite{h,t**}
Let $(X, ||\cdot||)$ be a normed linear space and $I\subset 2^\mathbb{N}$ be a non-trivial ideal. A sequence $\{x_n\}_{n\in\mathbb{N}}$ of elements of $X$ is said to be $I$-convergent to $x\in X$ if for each $\varepsilon >0$ the set $A(\varepsilon)=\{n\in\mathbb{N}: ||x_n - x||\geq \varepsilon\}$ belongs to $I$. The element $x$ is here called the $I$-limit of the sequence $\{x_n\}_{n\in\mathbb{N}}$.
\end{defn}
It should be noted here that if $I$ is an admissible ideal then usual convergence in $X$ implies $I$-convergence in $X$.
If $I_d$ denotes the class of all $A\subset \mathbb{N}$ with $d(A)=0$. Then $I_d$ is non trivial admissible ideal and $I_d$ convergence coincides with the statistical convergence.

\begin{defn}\cite{4th4,h}
Let $(X, ||\cdot||)$ be a normed linear space and $I\subset 2^\mathbb{N}$ be a non-trivial ideal. A sequence $\{x_n\}_{n\in\mathbb{N}}$ in $X$ is said to be $I^*$-convergent to $x$ if there exists a set $M=\{m_1<m_2<\cdots<m_k<\cdots\}$ in $F(I)$ such that the sub sequence $\{x_{m_k}\}_{k\in \mathbb{N}}$ is convergent to $x$ i.e., $\displaystyle{\lim_{k\rightarrow \infty}} ||x - x_{m_k}|| =0$.
\end{defn}
It is seen in \cite{h} that $I^*$-convergence implies $I$-convergence. If an admissible ideal $I$ has the property (AP), then for a sequence $\{x_n\}_{n\in\mathbb{N}}$ in a normed linear space $X$, $I$-convergence implies $I^*$-convergence.
\begin{defn}\label{correction}\cite{R2}
Let $\{x_n\}_{n\in\mathbb{N}}$ be a sequence in a normed linear space $(X, ||\cdot||)$ and $r$ be a non-negative real number. Then $\{x_n\}_{n\in\mathbb{N}}$ is said to be rough convergent of roughness degree $r$ to $x$ or simply $r$-convergent to $x$, denoted by $x_n \xrightarrow{r} x$, if for all $\varepsilon >0$ there exists $N(\varepsilon)\in\mathbb{N}$ such that $n\geq N(\varepsilon)$ implies $|| x_n - x||< r +\varepsilon$ and $x$ is called rough limit of $\{x_n\}_{n\in\mathbb{N}}$ of roughness degree $r$.
\end{defn}
For $r=0$, the definition \ref{correction} reduces to definition of usual convergence of sequences. Here $x$ is called the $r$-limit point of $\{x_n\}_{n\in\mathbb{N}}$, which is usually no more unique (for $r>0$). So we have to consider the so called $r$-limit set (or shortly $r$-limit) of $\{x_n\}_{n\in\mathbb{N}}$ defined by $LIM^r x_n :=\{ x\in X : x_n \xrightarrow{r} x\}$. A sequence $\{x_n\}_{n\in\mathbb{N}}$ is said to be $r$-convergent if $LIM^r x_n \neq \phi$. In this case, $r$ is called a rough convergence degree of $\{x_n\}_{n\in\mathbb{N}}$.
\begin{defn}\cite{R7}
A sequence $\{x_n\}_{n\in\mathbb{N}}$ in a normed linear space $(X, ||{\cdot}||)$ is said to be rough $I$-convergent to $x$, denoted by $x_n \xrightarrow{r-I} x$ provided that $\{n\in\mathbb{N}: ||x_n - x||\geq r +\varepsilon\}\in I$ for every $\varepsilon >0$.
\end{defn}
Here $r$ is called the roughness degree. If we take $r=0$, then the definition of rough $I$-convergence reduces to $I$-convergence. In general rough $I$-limit of a sequence $\{x_n\}$ may not be unique for the roughness degree $r >0$. So we have to consider the so called rough $I$-limit set of a sequence $\{x_n\}_{n\in\mathbb{N}}$ which is defined by $I-LIM^ r x_n :=\{x : x_n \xrightarrow{r-I} x\}$. A sequence $\{x_n\}_{n\in\mathbb{N}}$ is said to be rough $I$-convergent if $I-LIM^ r X_n\neq \phi$.
\begin{defn}\cite{5thpaper}
A sequence $\{x_n\}_{n\in \mathbb{N}}$ in a normed linear space $(X, || \cdot || )$ is said to be rough $I^*$-convergent of roughness degree $r$ to $x$ if there exists a set $M = \{m_1 < m_2< m_3 < \cdots < m_k < \cdots\}\in F(I)$ such that the sub sequence $\{x_{m_k}\}_{k\in \mathbb{N}}$ is rough convergent of roughness degree $r$ to $x$ i.e., for any $\varepsilon >0$ there exists a $N\in \mathbb{N}$ such that $|| x_{m_k} - x || < r + \varepsilon $ for all $k \geq N$ and we write $x_n \xrightarrow{r-I^*} x$.
\end{defn}
Here $x$ is called the rough $I^*$-limit of the sequence $\{x_n\}_{n\in \mathbb{N}}$ of roughness degree $r$. For $r=0$, we have the definition of ordinary $I^*$-convergence of sequences in normed linear spaces. Obviously rough $I^*$-limit of a sequence in normed linear spaces is not unique. Therefore we have to consider the rough $I^*$-limit set of the sequence $\{x_n\}_{n\in\mathbb{N}}$ defined as follows: $I^*-LIM^r x_n = \{x\in X :x_n \xrightarrow{r - I^*} x\}$. 
\begin{defn}\label{def1}\cite{p5/1}
A function $f: \mathbb{R}\mapsto \mathbb{R}$ is said to be $I$-continuous at a point $x_0\in \mathbb{R}$, if $I-\displaystyle{\lim_{n\to\infty}} x_n = x_0\implies I-\displaystyle{\lim_{n\to\infty}} f(x_n)= f(x_0)$.
\end{defn}
If $f$ is $I$-continuous at each point of a set $M\subset \mathbb{R}$, then $f$ is called $I$-continuous on the set $M$.\\
Also it should be noted in \cite{p5/1} that $I$-continuity coincides with ordinary continuity for every admissible ideal $I$. Also if $f$ and $g$ are $I$-continuous at $x_0$, then $f + g$ and $fg$ are $I$-continuous at $x_0$.

\begin{defn}\label{def2}\cite{p5/1}
Let $I_1$ and $I_2$ be two admissible ideals. A function $f: [a, b]\mapsto \mathbb{R}$ is said to be $(I_1, I_2)$-continuous at $x_0(\in [a, b])$ if $I_1-\displaystyle{\lim_{n\to\infty}} x_n = x_0\implies I_2-\displaystyle{\lim_{n\to\infty}} f(x_n)= f(x_0)$ for every sequence $\{x_n\}$.
\end{defn}
A function $f$ is said to be $(I_1, I_2)$-continuous on $[a, b]$ if it is $(I_1, I_2)$-continuous at each $x\in [a, b]$.

We now give the Definition of $I^*$-continuity from \cite{p5/1} as follows:
\begin{defn}\cite{p5/1}
A function $f$ is said to be $I^*$-continuous at $x_0$ if $I^*-\displaystyle{\lim_{n\to\infty}} x_n= x_0\implies I^*-\displaystyle{\lim_{n\to\infty}} f(x_n)=f(x_0)$ for every sequence $\{x_n\}$.
\end{defn}
The relationship between continuity and $I^*$-continuity of real valued function $f$ is as follows:
\begin{thm}\cite{p5/1}
If the ideal $I$ has the property $(AP)$. Then $f$ is $I^*$-continuous at $x_0$ if and only if $f$ is continuous at $x_0$.
\end{thm}
\section{Rough Continuity}

\begin{defn}\label{2}
Let $D\subset \mathbb{R}$. A function $f: D(\subset \mathbb{R})\mapsto \mathbb{R}$ is said to be rough continuous at a point $x\in D$ of roughness degree $r_x$, if for every sequences $\{x_n\}\subset D$ converging to $x$, there exists a non-negative real number $r_x$ such that the sequence $\{f(x_n)\}$ is rough convergent of roughness degree $r_x$ to $f(x)$. If $\rho=\sup\{r_x: x\in D\}$ exists finitely then $f$ is called rough continuous on $D$ of roughness degree $\rho$. 
\end{defn}

\begin{rem}
If a function is continuous at a point $x$ then obviously it is rough continuous at the point $x$ of roughness degree zero  i.e., $r_x=0$, but the converse may not be true which can be seen from the next two examples. Thus if we denote the set of all real valued continuous function by $C(f)$ and the set of  real valued rough continuous function of some roughness degree $R$ by $RC(f)$ then we have $C(f) \subset RC(f)$.
\end{rem}
\begin{exmp}
Let us consider the function defined by $f(x)=\begin{cases}
1, \:\: \text{if}\: x \:\text{is rational}\\
0,\:\: \text{if}\: x \:\text{is irrational}
\end{cases}$. Then it is easy to see that $f$ is rough continuous through out $\mathbb{R}$ of roughness degree $\rho=1$. For, suppose $x=a$ be arbitrary and let $\{x_n\}$ be a sequence converging to $a$. Then for any $\varepsilon>0$, we have either $|f(x_n)-f(a)|<1+\varepsilon$ or $|f(x_n)-f(a)|<0+\varepsilon$ for all $n\in\mathbb{N}$. Since this is true for any sequence $\{x_n\}$ converging to $a$, thus the function considered above is rough continuous at $a$ of roughness degree $r_a=1$. Again as the point $a$ is  chosen arbitrarily therefore the function $f$ is rough continuous throughout $\mathbb{R}$ of roughness degree  $\rho=1$ .
\end{exmp}
\begin{exmp}\label{example}
Let us consider the function $f: \mathbb{R}\mapsto \mathbb{R}$ such that $f(x)= [x]$. Obviously this function is not continuous for all integer values of $x\in \mathbb{R}$. But this function is rough continuous throughout $\mathbb{R}$ including all integers. For, let $\{x_n\}$ be a sequence which is convergent to an integer $x_0$. Then for arbitrary $\varepsilon >0$ there is a $N\in\mathbb{N}$ such that either $|f(x_n) - f(x_0)|< 1 + \varepsilon $ or $|f(x_n) - f(x_0)|< 0 + \varepsilon$  for all $n\geq N$. Hence $f$ is rough continuous at $x_0$ of roughness degree $r_{x_0}=1$. Also since the integer $x_0$ is chosen arbitrarily, it follows that the function considered here is rough continuous of roughness degree $\rho=1$ throughout $\mathbb{R}$ although it is not continuous at all integer points.
\end{exmp}
\begin{rem}\label{fg}
From the above Example \ref{example}, it is clear that if we consider the function $f^2(x)=f(x).f(x)=[x].[x]=[x]^2$, then it is easy to see that $f^2$ is rough continuous at the integer point $0$ and $\pm 1$ of roughness degree $1$ and the roughness degree becomes lager than $1$ at any integer points other than $0$ and $\pm 1$. Hence $f^2$ is rough continuous at every integer points other than $0$ and $\pm 1$ of roughness degree grater than $1$ although $f(x)=[x]$ is rough continuous at every integer points of roughness degree $1$. Hence the roughness degree of the function $f^2(x)=[x]^2$ is not the square of the roughness degree of the function $f$.
\end{rem}

\begin{thm}\label{th2}
A real valued bounded function is always rough continuous for some roughness degree.
\begin{proof}
Let $ D\subset \mathbb{R}$ and $f:D\mapsto \mathbb{R}$ be a bounded function. Let $\{x_n\}$ be a sequence in $D(\subset\mathbb{R})$ such that it is convergent to $x_0\in D$. Again since the function $f$ is bounded, so $\{f(x_n)\}$ is a bounded sequence in $\mathbb{R}$. Let $M$ be an upper bound of $f$.  So for any sequence $\{x_n\}$ converging to $x_0$, $|f(x_n) - f(x_0)|\leq |f(x_n)| +|f(x_0)|\leq 2M< 2M +\varepsilon$ for $\varepsilon>0$ and for all $n\in\mathbb{N}$. So it follows that $f$ is rough continuous at $x_0$.
\end{proof}
\end{thm}
Converse of the Theorem \ref{th2} is also true if we take $D$ as a closed and bounded interval of $\mathbb{R}$. Thus we have the following theorem.
\begin{thm}
If a function $f: D(\subset \mathbb{R})\mapsto \mathbb{R}$ is rough continuous of roughness degree $\rho$ and $D$ is a closed and bounded interval, then the function $f$ is bounded.
\end{thm}
\begin{proof}
Suppose that the function $f: D(\subset)\mathbb{R}\mapsto\mathbb{R}$ is rough continuous of roughness degree $\rho$. Let $D=[a, b]$. If possible let the function $f$ is not bounded. Hence for every $n\in\mathbb{N}$ we have a $x_n\in D$ such that $|f(x_n)|> n$. Now as $D$ is bounded so is the sequence $\{x_n\}_{n\in\mathbb{N}}$. Hence by Bolzano-Weierstrass Theorem, this sequence has a convergent sub sequence $\{x_{n_k}\}_{k\in\mathbb{N}}$. Since $D$ is closed  so let the sub sequence $\{x_{n_k}\}_{k\in\mathbb{N}}$ converges to some $x\in D$. Since $a\leq x_{n_k}\leq b$, therefore taking limit as $k\to \infty$ we get $a \leq x\leq b$. As $x\in D$, so let $|f(x)|=M $. Since $f$ is rough continuous of roughness degree $\rho$ on $D$, therefore the sequence $\{f(x_{n_k})\}$ is rough convergent of roughness degree $\rho$ to $f(x)$. Thus for any $\varepsilon >0$ there exists a $N\in\mathbb{N}$ such that $|f(x_{n_k}) - f(x)|< \rho +\varepsilon $ for all $k\geq N$. Therefore $|f(x_{n_k})|=|f(x_{n_k}) - f(x) + f(x)|\leq \rho +\varepsilon + |f(x)|= \rho +\varepsilon +M$ for all $k\geq N$. Take a natural number $r$ such that $n_r > \rho +\varepsilon +M$, and let $p=\max\{r, N\}$. Then $|f(x_{n_k})|\leq \rho +\varepsilon+M< n_r\leq n_p$ for all $k\geq p$, in particular $|f(x_{n_p})|< \rho+\varepsilon +M<n_r\leq n_p$. But $|f(x_{n_p})|>n_p$. This leads to a contradiction. Therefore $f$ is bounded.
\end{proof}

\begin{prop}
If functions $f$ and $g$ are rough continuous of roughness degree $r_{1_{_{x_0}}}$ and $r_{2_{_{x_0}}}$ respectively at $x_0$ then the function $f+g$ is also rough continuous at $x_0$ of roughness degree $r_{1_{_{x_0}}} + r_{2_{_{x_0}}}$.
\end{prop}
\begin{proof}
Let $f$ and $g$ be functions rough continuous at $x_0$ of roughness degree $r_{1_{_{x_0}}}$ and $r_{2_{_{x_0}}}$ respectively. Let $\varepsilon>0$ be arbitrary and $\{x_n\}$ be a sequence converging to $x_0$. Then according to our assumption and by the definition of rough continuity we have $f(x_0)\in LIM^ {r_{1_{_{x_0}}}} f(x_n)$ and $g(x_0)\in LIM^ {r_{2_{_{x_0}}}} g(x_n)$. Thus there exists $N_1, N_2\in\mathbb{N}$ such that $|f(x_n) - f(x_0)|< r_{1_{_{x_0}}} +\frac{\varepsilon}{2}$ and $|g(x_n) - g(x_0)|< r_{2_{_{x_0}}}+\frac{\varepsilon}{2}$ for all $n\geq N_1$ and $n\geq N_2$ respectively. Let $N=\max\{N_1, N_2\}$. Now $|(f +g)(x_n) -(f+g)(x_0)|=|f(x_n) - f(x_0) + g(x_n) - g(x_0)|\leq |f(x_n) - f(x_0)| + |g(x_n) - g(x_0)|< r_{1_{_{x_0}}} +\frac{\varepsilon}{2} + r_{2_{_{x_0}}} +\frac{\varepsilon}{2}=r_{1_{_{x_0}}} + r_{2_{_{x_0}}} +\varepsilon$ for all $n\geq N$. Since $\varepsilon>0$ is arbitrary, therefore it follows that $f+g$ is rough continuous at $x_0$ of roughness degree $r_{1_{_{x_0}}} + r_{2_{_{x_0}}}$. 
\end{proof}
\begin{thm}
If a function $f: D(\subset \mathbb{R})\mapsto \mathbb{R}$ is rough continuous at an arbitrary point $x_0\in D$ of roughness degree $r_{x_0}$, then the function $(cf)(x)= c f(x)$, where $c$ is a  non zero real number is rough continuous at point $x_0$ of roughness degree $|c|r_{x_0}$.
\end{thm}
\begin{proof}
Let a function $f: D(\subset \mathbb{R})\mapsto \mathbb{R}$ be rough continuous at an arbitrary point $x_0\in D$ of roughness degree $r_{x_0}$. Suppose $\{x_n\}$ be a sequence converging to $x_0$ and $c$ be a non zero real number. Then for an arbitrary $\varepsilon>0$ there exists a $N\in\mathbb{N}$ such that for the sequence $\{f(x_n)\}$ we have $|f(x_0) - f(x_n)|<r_{x_0} +\frac{\varepsilon}{|c|} $ for all $n\geq N$. Now $|cf(x_0) - cf(x_n)|=|c||f(x_0) - f(x_n)|< |c| (r_{x_0} +\frac{\varepsilon}{|c|})=|c|r_{x_0} +\varepsilon$ for all $n\geq N$. Since $\varepsilon  >0$ is arbitrary, therefore $cf$ is rough continuous at $x_0$ of roughness degree $|c|r_{x_0}$.
\end{proof}

\begin{cor}
If $f: D(\subset \mathbb{R})\mapsto \mathbb{R}$ is rough continuous at $x_0\in D$ of roughness degree $r_{x_0}$, then the function $\frac{1}{f}$ is rough continuous at $x_0$, provided that $f(x_0)$ is non zero.
\end{cor}
\begin{proof}
If possible let $f: D(\subset \mathbb{R})\mapsto \mathbb{R}$ is rough continuous at $x_0\in D$ of roughness degree $r_{x_0}$ and $\{x_n\}$ be a sequence converging to $x_0\in D$. Also let $f(x_0)=L(\neq 0)$. Now according to our assumption and by the definition of rough continuity, the sequence $\{f(x_n)\}$ is rough convergent of roughness degree  $r_{{x_0}}$ to $f(x)$. Also since the sequence $\{f(x_n)\}$ has a non empty rough limit set therefore $\{f(x_n)\}$ is bounded, so let $|f(x_n)|\leq M$ for all $n\in\mathbb{N}$. Let $\varepsilon >0$ be arbitrary. Then  by the definition of rough continuity of $f$ there exists $N\in\mathbb{N}$ such that $|f(x_n) - f(x_0)|<r_{x_0} +\varepsilon$ for all $n\geq N$. Now, $|\frac{1}{f(x_n)} -\frac{1}{f(x_0)}|=\frac{|f(x_n) - f(x_0)|}{|f(x_0)f(x_n)|}<\frac{r_{x_0} +\varepsilon}{|LM|}$ for all $n\geq N$. Hence it follows that $f$ is rough continuous at $x_0$.
\end{proof}
\begin{note}
Note that the roughness degree for rough continuity of $\frac{1}{f}$ depends on the value of $f$ at $x_0$ and on $M$.
\end{note}
Connectedness of a set in case of continuous image is preserved. But this is not true in case of rough continuity. We will justify our claim in the following remark.
\begin{rem}
Let us consider the function $f: [1,3]\mapsto \mathbb{R}$ defined by $f(x)=[x]$. As we have seen in the example \ref{example} that this function is rough continuous on $[1,3]$ of roughness degree $\rho=1$. But the image of the function is not an interval. Thus connectedness may not be preserved by the rough continuity unlike continuity.
\end{rem}
\begin{exmp}
 Let us consider the function $f$ on $D=[-1, 1]$by $f(x)=\begin{cases}
 0, x= -1\\
 2x, x\in (-1, 1]
 \end{cases}$. Now obviously $D$ is a compact subset of $\mathbb{R}$. Also it is easy to see that $f$ is rough continuous on $D$, but $f(D)= \{0\} \cup (-2,2]=(-2, 2]$. Therefore $f(D)$ is not closed set in $\mathbb{R}$ and hence it is not compact. Thus image of a compact set under rough continuous function is not necessarily a compact set.
\end{exmp}
A constant function on $D\subset \mathbb{R}$ is continuous and so rough continuous of roughness degree $\rho=0$. On other hand if a function on $D$ is rough continuous of roughness degree $\rho=0$ then it can not be constant. 
\begin{thm}
If the roughness degree $\rho$ of a function $f$ is non zero, then a constant function can not be rough continuous of roughness degree $\rho$.
\end{thm}
\begin{proof}
Suppose that $f: D(\subset \mathbb{R})\mapsto \mathbb{R}$ be a function such that it is rough continuous of roughness degree $\rho$, where $\rho$ is non zero. Since $\rho$ is non zero therefore there exists a non zero $R_{x_0}(\leq \rho)$ such that $f$ is rough continuous at some point $x_0\in D$ of roughness degree $R_{x_0}$. Now as $f$ is rough continuous at point $x_0$ of roughness degree $R_{x_0}$, therefore for any sequence $\{x_n\}$ converging to $x_0$, the sequence $\{f(x_n)\}$ is rough convergent of roughness degree $R_{x_0}$ to $f(x_0)$. Thus for a $\varepsilon>0$ we have a $N\in\mathbb{N}$ such that $|f(x_n) - f(x_0)|< R_{x_0} + \varepsilon$ for all $n\geq N$ $\rightarrow (i)$. Now since $R_{x_0}$ is non zero, therefore it follows from $(i)$ that $f$ can not be constant.   
\end{proof}
\section{Rough $I$-continuity}
Throughout this section we consider the ideal of subsets of $\mathbb{N}$, the set of all natural numbers.
\begin{defn}
A function $f: D(\subset \mathbb{R})\mapsto\mathbb{R}$ is said to be rough $I$-continuous of roughness degree ${r_{I}}_x$ at a point $x\in D$ if for every sequences $\{x_n\}$ $I$-converging to $x$, there exists a non-negative real number ${r_{I}}_x$ such that the sequence $\{f(x_n)\}$ is rough $I$-convergent of roughness degree ${r_{I}}_x$ to $f(x)$. Now if $\rho_{I}=\sup\{{r_{I}}_x: x\in D\}$ exists finitely, we then call $f$ is rough $I$-continuous on $D$ of roughness degree $\rho_I$.

\end{defn}
 If the $I$-roughness degree ${r_{I}}_x$ of $f$ at a point $x\in D$ becomes zero, the above definition reduced to the definition of ordinary $I$-continuity of the function at $x$. Obviously, for an admissible ideal $I$, rough continuity of a function $f$ implies that it is also rough $I$-continuous.
\begin{thm}
Let $I$ be a non trivial admissible ideal on $\mathbb{N}$. If a function $f: D(\subset\mathbb{R})\mapsto\mathbb{R}$ is rough $I$-continuous at a point $x\in D$, then it is also rough continuous at $x$.
\end{thm}
\begin{proof}
Let $I$ be an non trivial admissible ideal on $\mathbb{N}$ and $f: D\mapsto \mathbb{R}$ be rough $I$-continuous at $x\in D$ of some non zero roughness degree. If possible let the function $f: D(\subset \mathbb{R})\mapsto\mathbb{R}$ be not rough continuous at $x\in D$. Then there are a sequence $\{x_n\}$ converging to $x$ but there does not exists a non negative real number $r_x$ for which the  sequence $\{f(x_n)\}$ is rough convergent of roughness degree $r_x$ to $f(x)$. Now since $I$ is an admissible ideal therefore any sequence $\{x_n\}$ which convergent to $x$ is also $I$-convergent to $x$. Since for the sequence $\{x_n\}$ there does not exists a non negative real number $r_x$ for which the sequence $\{f(x_n)\}$ is rough convergent of roughness degree $r_x$ to $f(x)$, therefore there exists a $\varepsilon>0$ for which the set $\{n\in\mathbb{N}: |f(x_n) -f (x)|\geq r_x +\varepsilon\}=\mathbb{N}$ and as $I$ is a non trivial ideal so $\mathbb{N}\notin I$, therefore $f(x)$ is not rough $I$-convergent of roughness degree $r_x$ to $f(x)$, a contradiction. Hence the proof follows.
\end{proof}
\begin{defn}\label{defnII}
Let $I_1$ and $I_2$ be two admissible ideals. A function $f: D(\subset\mathbb{R})\mapsto\mathbb{R}$ is said to be rough $(I_1, I_2)$-continuous at $x\in D$, if for every sequence $\{x_n\}$ $I_1$-converging to $x$, there exists a non-negative real number $r_{(I_1.I_2)_x}$ such that $\{f(x_n)\}$ is rough $I_2$-convergent of roughness degree $r_{(I_1.I_2)_x}$ to $f(x)$. Now define $\rho_{(I_1, I_2)}=\sup\{r_{(I_1.I_2)_x}: x\in D\}$, if $\rho_{(I_1, I_2)}$ exists finitely, we then call $f$ is rough $(I_1,I_2)$ continuous of roughness degree $\rho_{(I_1, I_2)}$ throughout $D$.  
\end{defn}
Clearly if the $(I_1, I_2)$-roughness degree of $f$ at any point $x\in D$ becomes zero then rough $(I_1,I_2)$-continuity coincides with usual $(I_1, I_2)$-continuity of the function $f$ at $x$. Also it should be noted that the above definition reduces to rough $I$-continuity if $I_1=I_2=I$.
\begin{thm}
Let $I_1$ and $I_2$ be two admissible ideals on $\mathbb{N}$ such that $I_1\subset I_2$. Then a function $f: D(\subset\mathbb{R})\mapsto\mathbb{R}$ is rough $(I_1, I_2)$-continuous at a point $x\in D$ if and only if $f$ is rough $I_{fin}$-continuous at $x$, where $I_{fin}$ is the ideal of all finite subsets of $\mathbb{N}$.
\end{thm}
\begin{proof}
Let $\rho_{I_{fin}}$ be a non negative real number and the function $f: D(\subset\mathbb{R})\mapsto\mathbb{R}$ be rough $I_{fin}$-continuous at $x\in D$. Therefore, for every sequence $\{x_n\}$ which is $I_{fin}$-convergent to $x$ there exists a non negative real number $r_{{I_{fin}}_x}$ such that $\{f(x_n)\}$ is rough $I_{fin}$-convergent to $f(x)$ of roughness degree $r_{{I_{fin}}_x}$. Also let $I_1$ and $I_2$ be two ideals such that $I_1\subset I_2$. Let $\{x_n\}$ be $I_1$-convergent to $x$. Since $f$ is rough $I_{fin}$ continuous, it is rough continuous at $x$ and so it is rough $I_1$-continuous. So there exists a non negative real number $r_x$ such that $\{f(x_n)\}$ is rough $I_1$-convergent of roughness degree $r_x$ to $f(x)$. Since $I_1\subset I_2$ it follows that $\{f(x_n\}$ is rough $I_2$-convergent to $f(x)$ of roughness degree $r_x$.\\

Conversely, let $f$ be not rough $I_{fin}$-continuous at $x$. So there is a sequence $\{x_n\}$, $I_{fin}$-convergent to $x$ but $\{f(x_n)\}$ is not rough $I_{fin}$-convergent to $f(x)$. So for each non negative $r$, a suitable $\varepsilon >0$ be chosen such that the inequality $|f(x_n) - f(x_0)|\geq r +\varepsilon$ holds for $n=1, 2,\cdots$. Then $\{x_n\}$ is $I_1$-convergent to $x$ but $\{f(x_n\}$ is not rough $I_2$-convergent to $f(x)$, since the set $A(\varepsilon)=\{n\in\mathbb{N}: |f(x_n) - f(x)|\geq r +\varepsilon\}=\mathbb{N}\notin I_2$ for every non negative real $r$. So $f$ is not $(I_1, I_2)$ continuous at $x$. Hence if $f$ is $(I_1, I_2)$ rough continuous at $x$, it is rough $I_{fin}$ continuous at $x$.
\end{proof}

\begin{thm}\label{example}
Let $I_1$ and $I_2$ be two ideals on $\mathbb{N}$ such that $I_1\setminus I_2\neq \phi$ with the condition that there exists some $A\in I_1\setminus I_2$ with $\mathbb{N}\setminus A\notin I_2$. If the function $f: D(\subset \mathbb{R})\mapsto \mathbb{R}$ is constant then $f$ is rough $(I_1, I_2)$-continuous of any non zero roughness degree on $D$.
\end{thm}
\begin{proof}
Suppose that the function $f: D(\subset \mathbb{R})\mapsto\mathbb{R}$ is a constant function and let $x\in D$. Let a sequence $\{x_n\}$ in $D$ be such that it is $I_1$-convergent to the point $x$ in $D$. Now since $f$ is constant, for any $\varepsilon>0$, we have $\{n\in\mathbb{N}: |f(x_n) - f(x)|\geq r+\varepsilon\}=\phi\in I_2$ where $r$ is any non negative number. Hence $f$ is rough $(I_1, I_2)$-continuous at $x$ of roughness degree $r_{{(I_1, I_2)}_x}=r\geq 0$.
\end{proof}
The converse of the theorem \ref{example} is not true in general as shown in the following example.
\begin{exmp}
Let us consider the function $f$ defined as $f(x)=\begin{cases}
1, \:\: \text{if}\: x \:\text{is rational}\\
0,\:\: \text{if}\: x \:\text{is irrational}
\end{cases}$. Now let us take two ideals $I_1$ and $I_2$ such that $I_1\setminus I_2\neq \phi$. Let $x\in\mathbb{R}$ be arbitrary and $\{x_n\}$ be a sequence $I_1$-convergent to $x$. Now as $|f(x_n) - f(x)|\leq 1< 1+\varepsilon$ for all $n\in\mathbb{N}$ and for any $\varepsilon >0$, therefore the set $\{n\in\mathbb{N}: |f(x_n) - f(x)|\geq 1 +\varepsilon\}=\phi\in I_2$. Hence $f$ is rough $(I_1, I_2)$-continuous of roughness degree $1$ although it is not a constant function.
\end{exmp}
\begin{defn}
A function $f: D(\subset \mathbb{R})\mapsto \mathbb{R}$ is said to be rough $I^*$-continuous at a point $x\in D$ if for every sequence $\{x_n\}$, $I^*$-converging to $x$, there exists a non negative real number $r_{{I^*}_x}$ for which  the sequence $\{f(x_n)\}$ is rough $I^*$-convergent of roughness degree $r_{{I^*}_x}$ to $f(x)$. Now if $\rho_{I^*}=\{r_{{I^*}_x}: x\in D\}$ exists finitely, then $f$ is called rough $I^*$-continuous of roughness degree $\rho_{I^*}$ throughout $D$. 
\end{defn}
Clearly, if for any point $x\in D$ the roughness degree $\rho_{I^*}$ becomes zero, then the rough $I^*$-continuity becomes $I^*$-continuity of the function $f$ at $x$.
\begin{thm}\label{roughI*}
If a function $f: D(\subset \mathbb{R})\mapsto\mathbb{R}$ is rough $I^*$-continuous at a point $x_0\in D$, then it is also rough $I$-continuous at $x_0$.
\end{thm}
\begin{proof}
Since for a sequence $\{x_n\}$ of real numbers $I^*$-convergence implies $I$-convergence and rough $I^*$-limit of $\{x_n\}$ implies that it is also a rough $I$-limit, so the theorem follows immediately.
\end{proof}
The converse of the Theorem \ref{roughI*} is also true if we have an ideal $I$ which has the property (AP). 
\begin{thm}
Let $I$ be an ideal such that it has the property (AP). Then if a function $f: D(\subset\mathbb{R})\mapsto\mathbb{R}$ is rough $I$-continuous at a point $x_0\in D$, then it is also rough $I^*$-continuous at $x_0$.
\end{thm}
\begin{proof}
We know that if an ideal $I$ has the property (AP) then for a sequence $\{x_n\}$, $I$-convergence implies $I^*$-convergence. Again if an ideal $I$ has the property (AP), then rough $I$-limit of $\{x_n\}$ of some roughness degree $r$ is also a rough $I^*$-limit of $\{x_n\}$ of same roughness degree $r$. So the proof follows.
\end{proof}

\end{document}